\documentclass[12pt,oneside,reqno]{amsart}

\usepackage{amsmath,amssymb,amsthm,textcomp}
\usepackage{amsfonts,graphicx}
\usepackage[mathscr]{eucal}
\usepackage{color}
\usepackage{csquotes}
\usepackage[backend=bibtex,%
firstinits=true,%
doi=false,%
isbn=true,%
url=false,%
maxnames=99]{biblatex}%

\AtEveryBibitem{\clearfield{issn}}
\AtEveryCitekey{\clearfield{issn}}
\numberwithin{equation}{section}
\DeclareNameAlias{sortname}{last-first}
\theoremstyle{definition}

\numberwithin{equation}{section}

\newcommand{\ncom}{\newcommand}

\ncom{\beq}{\begin{equation}}
\ncom{\eeq}{\end{equation}}
\ncom{\bea}{\begin{eqnarray*}}
\ncom{\eea}{\end{eqnarray*}}
\ncom{\beqa}{\begin{eqnarray}}
\ncom{\eeqa}{\end{eqnarray}}
\ncom{\nno}{\nonumber}
\ncom{\non}{\nonumber}
\ncom{\ds}{\displaystyle}
\ncom{\half}{\frac{1}{2}}
\ncom{\mbx}{\makebox{.25cm}}
\ncom{\hs}{\mbox{\hspace{.25cm}}}
\ncom{\rar}{\rightarrow}
\ncom{\Rar}{\Rightarrow}
\ncom{\noin}{\noindent}
\ncom{\bc}{\begin{center}}
\ncom{\ec}{\end{center}}
\ncom{\sz}{\scriptsize}
\ncom{\rf}{\ref}
\ncom{\s}{\sqrt{2}}
\ncom{\sgm}{\sigma}
\ncom{\Sgm}{\Sigma}
\ncom{\psgm}{\sigma^{\prime}}
\ncom{\dt}{\delta}
\ncom{\Dt}{\Delta}
\ncom{\lmd}{\lambda}
\ncom{\Lmd}{\Lambda}
\ncom{\Th}{\Theta}
\ncom{\e}{\eta}
\ncom{\eps}{\epsilon}
\ncom{\pcc}{\stackrel{P}{>}}
\ncom{\lp}{\stackrel{L_{p}}{>}}
\ncom{\dist}{{\rm\,dist}}
\ncom{\sspan}{{\rm\,span}}
\ncom{\re}{{\rm Re\,}}
\ncom{\im}{{\rm Im\,}}
\ncom{\sgn}{{\rm sgn\,}}
\ncom{\ba}{\begin{array}}
\ncom{\ea}{\end{array}}
\ncom{\hone}{\mbox{\hspace{1em}}}
\ncom{\htwo}{\mbox{\hspace{2em}}}
\ncom{\hthree}{\mbox{\hspace{3em}}}
\ncom{\hfour}{\mbox{\hspace{4em}}}
\ncom{\vone}{\vskip 2ex}
\ncom{\vtwo}{\vskip 4ex}
\ncom{\vonee}{\vskip 1.5ex}
\ncom{\vthree}{\vskip 6ex}
\ncom{\vfour}{\vspace*{8ex}}
\ncom{\norm}{\|\;\;\|}
\ncom{\integ}[4]{\int_{#1}^{#2}\,{#3}\,d{#4}}
\ncom{\vspan}[1]{{{\rm\,span}\{ #1 \}}}
\ncom{\dm}[1]{ {\displaystyle{#1} } }
\ncom{\ri}[1]{{#1} \index{#1}}

\newtheorem{theorem}{\bf Theorem}[section]
\newtheorem{remark}{\bf Remark}[section]

\newtheorem{lemma}{Lemma}[section]
\newtheorem{corollary}{Corollary}[section]
\newtheorem{example}{Example}[section]

\newtheoremstyle
    {remarkstyle}
    {}
    {11pt}
    {}
    {}
    {\bfseries}
    {:}
    {     }
    {\thmname{#1} \thmnumber{#2} }

\theoremstyle{remarkstyle}

\def\eps{\varepsilon}

\begin{document}
\title{The $I$-Function Distribution and its Extensions}
\author{P. Vellaisamy}
\address{P. Vellaisamy, Department of Mathematics,
 Indian Institute of Technology Bombay, Powai, Mumbai 400076, INDIA.}
 \email{pv@math.iitb.ac.in}
\author[Kuldeep Kumar Kataria]{K. K. Kataria}
\address{Kuldeep Kumar Kataria, Department of Mathematics,
 Indian Institute of Technology Bombay, Powai, Mumbai 400076, INDIA.}
 \email{kulkat@math.iitb.ac.in}
\thanks{The research of K. K. Kataria was supported by UGC, Govt. of India.}
\begin{abstract}
In this paper we introduce a new probability distribution on $(0,\infty)$, associated with the $I$-function, namely, the $I$-function distribution. This distribution generalizes several known distributions with positive support. It is also shown that the distribution of products, quotients and powers of independent $I$-function variates are $I$-function variates. Another distribution called the $I$-function inverse Gaussian distribution is also introduced and studied. 
\end{abstract}

\maketitle
\section{Introduction}
\setcounter{equation}{0}
Springer and Thompson \cite{Springer721} showed that the probability density functions (pdf's) of products and quotients of independent beta, gamma and central Gaussian random variables (rv's) can be expressed in terms of the Meijer $G$-function (see \cite{Mathai2010}, p. 21). Mathai and Saxena \cite{Mathai1439} discussed the distribution of the product of two independent rv's whose pdf can be expressed as the product of two $H$-functions; see \cite{Fox395}, \cite{Mathai2010} for definition and more details on the $H$-function. Carter and Springer \cite{Springer542} introduced the $H$-function distribution which includes, among others, the gamma, beta, Weilbull, chi-square, exponential and the half-normal distribution as particular cases. Inayat-Hussain \cite{Hussain4119} generalized the $H$-function, namely, to the $\overline{H}$-function. The generalized Riemann zeta function, the polylogarithmic function of complex order and the exact partition function of the Gaussian free energy model in statistical mechanics are some special cases of the $\overline{H}$-function, which are not particular cases of the $H$-function. Rathie \cite{Rathie297} introduced the $I$-function, which includes the $\overline{H}$-function as a special case. The $I$-function is represented by the following Mellin-Barnes type contour integral
\begin{equation}\label{1.2}
I(z)=I^{m,n}_{p,q}\Bigg[z\left|
\begin{matrix}
    (a_i,A_i,\alpha_i)_{1,p}\\ 
    (b_j,B_j,\beta_j)_{1,q}
  \end{matrix}
\right.\Bigg]=\frac{1}{2\pi i}\int_{C}\chi(s)z^{-s}\,ds,
\end{equation}
where
\begin{equation}\label{1.3}
\chi(s)=\frac{\prod_{i=1}^{n}\Gamma^{\alpha_i}\left(1-a_i-A_is\right)\prod_{j=1}^{m}\Gamma^{\beta_j}\left(b_j+B_js\right)}{\prod_{i=n+1}^{p}\Gamma^{\alpha_i}\left(a_i+A_is\right)\prod_{j=m+1}^{q}\Gamma^{\beta_j}\left(1-b_j-B_js\right)}.
\end{equation}
An equivalent definition can be obtained on substituting $w=-s$ in (\ref{1.2}). In the above definition, $z\neq0,m, n, p, q$ are integers satisfying $0 \leq m \leq q$ and $0 \leq n \leq p$ with $\alpha_i,A_i>0$ for $i=1,2,\ldots,p$ and $\beta_j,B_j>0$ for $j=1,2,\ldots,q$. Also, $a_i$'s and $b_j$'s are complex numbers such that no singularity of $\Gamma^{\beta_j}\left(b_j+B_js\right)$ coincides with any singularity of $\Gamma^{\alpha_i}\left(1-a_i-A_is\right)$. An empty product is to be interpreted as unity. The path of integration $C$, in the complex $s$-plane runs from $c-i\infty$ to $c+i\infty$ for some real number $c$ such that the singularity of $\Gamma^{\beta_j}\left(b_j+B_js\right)$ lie entirely to the left of the path and the singularity of $\Gamma^{\alpha_i}\left(1-a_i-A_is\right)$ lie entirely to the right of the path. The conditions for convergence of the integral involved in (\ref{1.2}) are as follows (see \cite{Rathie297}). Let
\begin{eqnarray*}
\nabla&=&\sum_{j=1}^q\beta_jB_j-\sum_{i=1}^p\alpha_iA_i,\\
\Omega&=&\sum_{i=1}^p\left(\frac{1}{2}-\operatorname{Re}(a_i)\right)\alpha_i-\sum_{j=1}^q\left(\frac{1}{2}-\operatorname{Re}(b_j)\right)\beta_j,\\
\Delta&=&\sum_{j=1}^m\beta_jB_j-\sum_{j=m+1}^q\beta_jB_j+\sum_{i=1}^n\alpha_iA_i-\sum_{i=n+1}^p\alpha_iA_i,
\end{eqnarray*}
where $m,n$, $a_i,A_i,\alpha_i$ and $b_j,B_j,\beta_j$ appear in the definition of the $I$-function. The $I$-function is analytic if $\nabla\geq0$ and the integral in (\ref{1.2}) converges absolutely if $|\arg(z)|<\Delta\frac{\pi}{2}$, where $\Delta>0$. Also, if $|\arg(z)|=\Delta\frac{\pi}{2}$ with $\Delta\geq0$, then it converges absolutely under the following conditions: (i) $\nabla=0$ and $\Omega<-1$, (ii) $|\nabla|\neq0$ with $s=\sigma+it$, where $\sigma,t\in\mathbb{R}$ and are such that for $|t|\rightarrow\infty$ we have $\Omega+\sigma\nabla<-1$. The $\overline{H}$-function follows as a particular case when $\alpha_i=1$ for $i=n+1,\ldots,p$ and $\beta_j=1$ for $j=1,2,\ldots,m$ in (\ref{1.2}). Recently, the extension of the $I$-function to several complex variables (see \cite{Kumari285}, \cite{Prathima}) and its applications to wireless communication were studied by several authors (see \cite{Ansari1}, \cite{Ansari394}, \cite{Minghua5578}). The effect of Marichev-Saigo-Maeda fractional operators on the $I$-function has been recently studied by Kataria and Vellaisamy (see \cite{Kataria}).

The main purpose of this paper is to introduce a new probability distribution based on the $I$-function, which we call the $I$-function distribution. In Section 2, some preliminary results are stated that will be used in subsequent sections. The Mellin and Laplace transforms of the $I$-function are evaluated in Section 3. In Section 4, the $I$-function distribution is introduced and some results related to the distribution of products, quotients and powers of independent $I$-function variates are derived. Also, the $I$-function inverse Gaussian distribution is introduced and studied in Section 5. 

\section{Preliminaries}
\setcounter{equation}{0}
The $I$-function is connected to the generalized hypergeometric function ${}_pF_q$, the Meijer's $G$-function, the generalized Wright function ${}_p\psi_q$, the Fox's $H$-function and the $\overline{H}$-function by the following relationships:
\begin{eqnarray*}
I^{1,p}_{p,q+1}\Bigg[-z\left|
\begin{matrix}
    (1-a_i,1,1)_{1,p}\\ 
    (0,1,1)&(1-b_j,1,1)_{1,q}
  \end{matrix}
\right.\Bigg]&=&\frac{\prod_{i=1}^p\Gamma(a_i)}{\prod_{j=1}^q\Gamma(b_j)}{}_pF_q\left((a_i)_{1,p};(b_j)_{1,q};z\right);\\
I^{m,n}_{p,q}\Bigg[z\left|
\begin{matrix}
    (a_i,1,1)_{1,p}\\ 
    (b_j,1,1)_{1,q}
  \end{matrix}
\right.\Bigg]&=&G^{m,n}_{p,q}\Bigg[z\left|
\begin{matrix}
    (a_i)_{1,p}\\ 
    (b_j)_{1,q}
  \end{matrix}
\right.\Bigg];\\
I^{1,p}_{p,q+1}\Bigg[-z\left|
\begin{matrix}
    (1-a_i,A_i,1)_{1,p}\\ 
    (0,1,1)&(1-b_j,B_j,1)_{1,q}
  \end{matrix}
\right.\Bigg]&=&{}_p\psi_q\Bigg[z\left|
\begin{matrix}
    (a_i,A_i)_{1,p}\\ 
    (b_j,B_j)_{1,q}
  \end{matrix}
\right.\Bigg];\\
I^{m,n}_{p,q}\Bigg[z\left|
\begin{matrix}
    (a_i,A_i,1)_{1,p}\\ 
    (b_j,B_j,1)_{1,q}
  \end{matrix}
\right.\Bigg]&=&H^{m,n}_{p,q}\Bigg[z\left|
\begin{matrix}
    (a_i,A_i)_{1,p}\\ 
    (b_j,B_j)_{1,q}
  \end{matrix}
\right.\Bigg];
\end{eqnarray*}
and
\begin{equation*}
I^{m,n}_{p,q}\Bigg[z\left|
\begin{matrix}
    (a_i,A_i,\alpha_i)_{1,n}&(a_i,A_i,1)_{n+1,p}\\ 
    (b_j,B_j,1)_{1,m}&(b_j,B_j,\beta_j)_{m+1,q}
  \end{matrix}
\right.\Bigg]=\overline{H}^{m,n}_{p,q}\Bigg[z\left|
\begin{matrix}
    (a_i,A_i,\alpha_i)_{1,n}&(a_i,A_i)_{n+1,p}\\ 
    (b_j,B_j)_{1,m}&(b_j,B_j,\beta_j)_{m+1,q}
  \end{matrix}
\right.\Bigg].
\end{equation*}
The following identities will be used later:
\begin{equation}\label{m2.8}
I^{1,0}_{0,1}\Bigg[z\left|
\begin{matrix}
     \\ 
    (b,B,1)
  \end{matrix}
\right.\Bigg]=\frac{1}{B}z^{\frac{b}{B}}e^{-z^{\frac{1}{B}}};\ \ \ \ 
I^{1,1}_{1,1}\Bigg[z\left|
\begin{matrix}
     (b-a+1,1,1)\\ 
    (b,1,1)
  \end{matrix}
\right.\Bigg]=\Gamma(a)z^b(1+z)^{-a};
\end{equation}
\begin{equation}
I^{0,1}_{1,1}\Bigg[z\left|
\begin{matrix}
     (b+1,1,1)\\ 
    (b,1,1)
  \end{matrix}
\right.\Bigg]=z^b,\ \ |z|>1;
\end{equation}
and
\begin{equation}\label{m2.10}
\Gamma(a+1)\delta^{a+b}I^{1,0}_{1,1}\Bigg[\frac{z}{\delta}\left|
\begin{matrix}
     (a+b+1,1,1)\\ 
    (b,1,1)
  \end{matrix}
\right.\Bigg]=z^b(\delta-z)^a,\ \ |z|<\delta.
\end{equation}
Also, from the definition of the $I$-function, the following properties are immediate.
\begin{eqnarray}
I^{m,n}_{p,q}\Bigg[\frac{1}{z}\left|
\begin{matrix}
    (a_i,A_i,\alpha_i)_{1,p}\\ 
    (b_j,B_j,\beta_j)_{1,q}
  \end{matrix}
\right.\Bigg]&=&I^{n,m}_{q,p}\Bigg[z\left|
\begin{matrix}
    (1-b_j,B_j,\beta_j)_{1,q}\\
    (1-a_i,A_i,\alpha_i)_{1,p}    
  \end{matrix}
\right.\Bigg];\label{2.8}\\
I^{m,n}_{p,q}\Bigg[z^\sigma\left|
\begin{matrix}
    (a_i,A_i,\alpha_i)_{1,p}\\ 
    (b_j,B_j,\beta_j)_{1,q}
  \end{matrix}
\right.\Bigg]&=&\frac{1}{\sigma}I^{m,n}_{p,q}\Bigg[z\left|
\begin{matrix}
    \left(a_i,\sigma^{-1}A_i,\alpha_i\right)_{1,p}\\ 
    \left(b_j,\sigma^{-1}B_j,\beta_j\right)_{1,q}
  \end{matrix}
\right.\Bigg],\ \ \sigma>0;\\
z^\sigma I^{m,n}_{p,q}\Bigg[z\left|
\begin{matrix}
    (a_i,A_i,\alpha_i)_{1,p}\\ 
    (b_j,B_j,\beta_j)_{1,q}
  \end{matrix}
\right.\Bigg]&=&I^{m,n}_{p,q}\Bigg[z\left|
\begin{matrix}
    (a_i+\sigma A_i,A_i,\alpha_i)_{1,p}\\ 
    (b_j+\sigma B_j,B_j,\beta_j)_{1,q}
  \end{matrix}
\right.\Bigg]\label{2.9}.
\end{eqnarray}
Note that whether a product of two $I$-functions is another $I$-function is still an open problem.
\section{The Mellin and Laplace transform of the $I$-Function}
\setcounter{equation}{0}
In this section, we evaluate the Mellin and Laplace transform of the $I$-function. The Mellin and Laplace transform of a continuous positive rv $X$ with pdf $f_X$ is defined by
\begin{equation}\label{2.1}
\mathcal{M}_{f_X}(s)=E[X^{s-1}]=\int_0^{\infty}x^{s-1}f_X(x)\,dx
\end{equation}
and
\begin{equation}\label{n2.1}
\mathcal{L}_{f_X}(r)=E[e^{-rX}]=\int_0^{\infty}e^{-rx}f_X(x)\,dx,
\end{equation}
respectively. Under some suitable restrictions \cite{Titchmarsh1937} on $\mathcal{M}_{f_X}(s)$, there exists an inversion integral
\begin{equation}\label{2.2}
f_X(x)=\frac{1}{2\pi i}\int_{c-i\infty}^{c+i\infty}x^{-s}\mathcal{M}_{f_X}(s)\,ds.
\end{equation}
The following results are due to Epstein \cite{Epstein370}.\\
\noindent (i) The Mellin transform of the positive rv $Y=aX$, $a>0$ is
\begin{equation}\label{2.3}
\mathcal{M}_{f_Y}(s)=a^{s-1}\mathcal{M}_{f_X}(s).
\end{equation}
\noindent (ii) The Mellin transform of the positive rv $Y=X^r$ is
\begin{equation}\label{2.4}
\mathcal{M}_{f_Y}(s)=\mathcal{M}_{f_X}(rs-r+1).
\end{equation}
\noindent (iii) Let $X_1,X_2$ be two positive continuous independent rv's. The Mellin transform of the rv $Y=X_1/X_2$ is
\begin{equation}\label{2.5}
\mathcal{M}_{f_Y}(s)=\mathcal{M}_{f_{X_1}}(s)\mathcal{M}_{f_{X_2}}(2-s).
\end{equation}
\noindent (iv) Let $X_1,X_2,\ldots,X_n$ be $n$ positive continuous independent rv's. Then the Mellin transform of the rv $Y=\prod_{j=1}^{n}X_j$ is
\begin{equation}\label{2.6}
\mathcal{M}_{f_Y}(s)=\prod_{j=1}^{n}\mathcal{M}_{f_{X_j}}(s).
\end{equation}
\subsection{The Mellin transform}
Assuming the convergence of the integral involved in definition of the $I$-function, for $\sigma>0$ and $\delta\neq0$, we have
\begin{eqnarray}\label{nn3.1}
I(\delta x^{\sigma})&=&I^{m,n}_{p,q}\Bigg[\delta x^{\sigma}\left|
\begin{matrix}
    (a_i,A_i,\alpha_i)_{1,p}\\ 
    (b_j,B_j,\beta_j)_{1,q}
  \end{matrix}
\right.\Bigg]\\
&=&\frac{1}{2\pi i}\int_{C}\chi(w)\delta^{-w}x^{-\sigma w}\,dw.\nonumber
\end{eqnarray}
Substituting $s=\sigma w$ in the above integral and using (\ref{2.2}), the Mellin transform of the $I$-function (\ref{nn3.1}) is
\begin{equation}\label{n3.1}
\mathcal{M}_{I(\delta x^{\sigma})}(s)=\frac{\chi(\sigma^{-1}s)}{\sigma\delta^{\sigma^{-1}s}}.
\end{equation}
The above result can be extended to obtain the Mellin transform of product of two  $I$-functions. We state the following useful result in the form of a lemma.
\begin{lemma}\label{l3.1} The Mellin transform, $\mathcal{M}_{I(\delta x^{\sigma})I(\eta x^\mu)}(s)$, of the product of
\begin{equation*}
I(\delta x^\sigma)=I^{m,n}_{p,q}\Bigg[\delta x^\sigma\left|
\begin{matrix}
    (a_i,A_i,\alpha_i)_{1,p}\\ 
    (b_j,B_j,\beta_j)_{1,q}
  \end{matrix}
\right.\Bigg]\ \ \ \ \ \ \mathrm{and}\ \ \ \ \ 
I(\eta x^\mu)=I^{k,l}_{u,v}\Bigg[\eta x^\mu\left|
\begin{matrix}
    (c_i,C_i,\gamma_i)_{1,u}\\ 
    (d_j,D_j,\rho_j)_{1,v}
  \end{matrix}
\right.\Bigg]
\end{equation*}
for $\sigma>0,\mu\geq0$ and $\delta,\eta\neq 0$ is equal to 
\begin{equation*}
\frac{1}{\sigma\delta^{\sigma^{-1}s}}I^{n+k,m+l}_{q+u,p+v}\Bigg[\frac{\eta}{\delta^{\sigma^{-1}\mu}}\left|
\begin{matrix}
    (c_i,C_i,\gamma_i)_{1,l}&(1-b_j-\sigma^{-1}B_js,\mu\sigma^{-1} B_j,\beta_j)_{1,q}&(c_i,C_i,\gamma_i)_{l+1,u}\\ 
    (d_j,D_j,\rho_j)_{1,k}&(1-a_i-\sigma^{-1}A_is,\mu\sigma^{-1} A_i,\alpha_i)_{1,p}&(d_j,D_j,\rho_j)_{k+1,v}
  \end{matrix}\right.\Bigg].
\end{equation*}
\end{lemma}
\begin{proof} From (\ref{2.1}), we have
\begin{eqnarray*}
\mathcal{M}_{I(\delta x^{\sigma})I(\eta x^\mu)}(s)&=&\int_0^{\infty}x^{s-1}I^{m,n}_{p,q}\Bigg[\delta x^{\sigma}\left|
\begin{matrix}
    (a_i,A_i,\alpha_i)_{1,p}\\ 
    (b_j,B_j,\beta_j)_{1,q}
  \end{matrix}
\right.\Bigg]
I^{k,l}_{u,v}\Bigg[\eta x^\mu\left|
\begin{matrix}
    (c_i,C_i,\gamma_i)_{1,u}\\ 
    (d_j,D_j,\rho_j)_{1,v}
  \end{matrix}
\right.\Bigg]\,dx
\\&=&\int_0^{\infty}x^{s-1}I^{m,n}_{p,q}\Bigg[\delta x^{\sigma}\left|
\begin{matrix}
    (a_i,A_i,\alpha_i)_{1,p}\\ 
    (b_j,B_j,\beta_j)_{1,q}
  \end{matrix}
\right.\Bigg]
\frac{1}{2\pi i}\int_{C}\chi^*(w)(\eta x^\mu)^{-w}\,dw\,dx,
\end{eqnarray*}
where
\begin{equation*}
\chi^*(w)=\frac{\prod_{i=1}^{l}\Gamma^{\gamma_i}\left(1-c_i-C_iw\right)\prod_{j=1}^{k}\Gamma^{\rho_j}\left(d_j+D_jw\right)}{\prod_{i=l+1}^{u}\Gamma^{\gamma_i}\left(c_i+C_iw\right)\prod_{j=k+1}^{v}\Gamma^{\rho_j}\left(1-d_j-D_jw\right)}.
\end{equation*}
We assume the conditions for absolute convergence of integral involved in above equation holds. By interchanging the order of integrals, we get
\begin{eqnarray*}
\mathcal{M}_{I(\delta x^{\sigma})I(\eta x^\mu)}(s)&=&\frac{1}{2\pi i}\int_{C}\chi^*(w)\eta^{-w}\int_0^{\infty}x^{s-\mu w-1}I^{m,n}_{p,q}\Bigg[\delta x^{\sigma}\left|
\begin{matrix}
    (a_i,A_i,\alpha_i)_{1,p}\\ 
    (b_j,B_j,\beta_j)_{1,q}
  \end{matrix}
\right.\Bigg]\,dx\,dw\\
&=&\frac{1}{2\pi i}\int_{C}\chi^*(w)\eta^{-w}\mathcal{M}_{I(\delta x^{\sigma})}(s-\mu w)\,dw.\\
&=&\frac{1}{\sigma\delta^{\sigma^{-1}s}}\frac{1}{2\pi i}\int_{C}\chi^*(w)\chi(\sigma^{-1}s-\sigma^{-1}\mu w)\left(\eta\delta^{-\sigma^{-1}\mu}\right)^{-w}\,dw.
\end{eqnarray*}
The proof is complete by using definition of the $I$-function (\ref{1.2}).
\end{proof}
\subsection{The Laplace transform}
\noindent  For simplicity, we use an equivalent definition of the $I$-function to determine its Laplace transform. Moreover, the conditions for absolute convergence of integral involved in the definition of $I(\delta x^{\sigma})$ is assumed. From (\ref{n2.1}), we have
\begin{eqnarray}\label{n3.2}
\mathcal{L}_{I(\delta x^{\sigma})}(r)&=&\int_0^{\infty}e^{-rx}I(\delta x^{\sigma})\,dx\nonumber\\
&=&\int_0^{\infty}\frac{e^{-rx}}{2\pi i}\int_{C}\chi(-s)(\delta x^{\sigma})^s\,ds\,dx\nonumber\\
&=&\frac{1}{2\pi i}\int_{C}\chi(-s)\delta^s\int_0^{\infty}x^{\sigma s}e^{-rx}\,dx\,ds\nonumber\\
&=&\frac{1}{r}\frac{1}{2\pi i}\int_{C}\Gamma(1+\sigma s)\chi(-s)\left(\frac{\delta}{r^{\sigma}}\right)^s\,ds\nonumber\\
&=&\frac{1}{r}I^{m,n+1}_{p+1,q}\Bigg[\frac{\delta}{r^{\sigma}}\left|
\begin{matrix}
    (0,\sigma,1)&(a_i,A_i,\alpha_i)_{1,p}\\ 
    (b_j,B_j,\beta_j)_{1,q}
  \end{matrix}
\right.\Bigg]\nonumber\\
&=&\frac{r^{\sigma-1}}{\delta}I^{m,n+1}_{p+1,q}\Bigg[\frac{\delta}{r^{\sigma}}\left|
\begin{matrix}
    (\sigma,\sigma,1)&(a_i+A_i,A_i,\alpha_i)_{1,p}\\ 
    (b_j+B_j,B_j,\beta_j)_{1,q}
  \end{matrix}
\right.\Bigg]\nonumber\ \ \ \ (\mathrm{using}\ (\ref{2.9}))\\
&=&\frac{r^{\sigma-1}}{\delta}I^{n+1,m}_{q,p+1}\Bigg[\frac{r^{\sigma}}{\delta}\left|
\begin{matrix}
    (1-b_j-B_j,B_j,\beta_j)_{1,q}\\
    (1-\sigma,\sigma,1)&(1-a_i-A_i,A_i,\alpha_i)_{1,p} 
      \end{matrix}
\right.\Bigg],
\end{eqnarray}
where the last step follows from (\ref{2.8}).

\section{The $I$-Function Distribution}
\setcounter{equation}{0}
Here, we introduce a new probability distribution associated with the $I$-function. Let $X>0$ be a rv with pdf
\begin{equation}\label{4.1}
f_X(x) =\left\{
	\begin{array}{ll}
	   k_0I^{m,n}_{p,q}\Bigg[\delta x^{\sigma}\left|
\begin{matrix}
    (a_i,A_i,\alpha_i)_{1,p}\\ 
    (b_j,B_j,\beta_j)_{1,q}
  \end{matrix}
\right.\Bigg], & \mbox{} x>0,\\\\
		0  & \mbox{} \mathrm{otherwise},
	\end{array}
\right.
\end{equation}
where $\sigma>0$, $\delta\neq0$ and $(a_i,A_i,\alpha_i)_{1,p},(b_j,B_j,\beta_j)_{1,q},m,n$ must conform to those restrictions in the definition of the $I$-function (\ref{1.2}), such that $f_X(x)\geq 0$. Also, $k_0$ is such that $\int_0^{\infty}f_X(x)\,dx=1$.
\subsubsection*{(a) Moments}
The $r$-th moment for the $I$-function distribution can be obtained using the Mellin integral transform as follows.
\begin{eqnarray*}
\mu'_r&=&\mathcal{M}_{f_X}(r+1)\\
&=&k_0\mathcal{M}_{I(\delta x^{\sigma})}(r+1)\\
&=&\frac{k_0\chi(\sigma^{-1}+\sigma^{-1}r)}{\sigma\delta^{(r+1)\sigma^{-1}}},
\end{eqnarray*}
using (\ref{n3.1}) and $\chi(.)$ is defined in (\ref{1.3}). The normalizing constant of the $I$-function distribution is obtained from $\mu'_0=1$ as
\begin{equation*}
k_0=\frac{\sigma\delta^{\sigma^{-1}}}{\chi(\sigma^{-1})}.
\end{equation*}
\subsubsection*{(b) The characteristic function}
The characteristic function of the $I$-function distribution is given by
\begin{eqnarray*}
\phi_X(t)&=&k_0\mathcal{L}_{I(\delta x^{\sigma})}(-it)\\
&=&\frac{k_0(-it)^{\sigma-1}}{\delta}I^{n+1,m}_{q,p+1}\Bigg[\frac{(-it)^{\sigma}}{\delta}\left|
\begin{matrix}
    (1-b_j-B_j,B_j,\beta_j)_{1,q}\\
    (1-\sigma,\sigma,1)&(1-a_i-A_i,A_i,\alpha_i)_{1,p} 
      \end{matrix}
\right.\Bigg],
\end{eqnarray*}
where the last step follows from (\ref{n3.2}). Hence, the moment generating function is
\begin{eqnarray*}
M_X(t)&=&\phi_X(-it)\\
&=&\frac{k_0(-t)^{\sigma-1}}{\delta}I^{n+1,m}_{q,p+1}\Bigg[\frac{(-t)^{\sigma}}{\delta}\left|
\begin{matrix}
    (1-b_j-B_j,B_j,\beta_j)_{1,q}\\
    (1-\sigma,1,1)&(1-a_i-A_i,A_i,\alpha_i)_{1,p} 
      \end{matrix}
\right.\Bigg].
\end{eqnarray*}
\subsubsection*{(c) The cumulative distribution function (cdf)}
The Mellin transform of the cdf $F_X(x)$ of a rv $X$ is connected to its pdf $f_X(x)$ by the following relationship (see \cite{Springer1979}, p. 99). 
\begin{equation}\label{n3.4}
\mathcal{M}_{1-F_X(x)}(s)=s^{-1}\mathcal{M}_{f_X}(s+1).
\end{equation}
Therefore the cdf of the $I$-function distribution is given by
\begin{equation*}
F_X(x)=1-\frac{k_0}{\sigma\delta^{\sigma^{-1}}}I^{m+1,n}_{p+1,q+1}\Bigg[\delta^{\sigma^{-1}} x\left|
\begin{matrix}
    (a_i+\sigma^{-1}A_i,\sigma^{-1}A_i,\alpha_i)_{1,p}&(1,1,1)\\ 
    (0,1,1)&(b_j+\sigma^{-1}B_j,\sigma^{-1}B_j,\beta_j)_{1,q}
  \end{matrix}
  \right.\Bigg].
\end{equation*}
And hence, the survival function is
\begin{equation*}
\overline{F}_X(x)=\frac{k_0}{\sigma\delta^{\sigma^{-1}}}I^{m+1,n}_{p+1,q+1}\Bigg[\delta^{\sigma^{-1}} x\left|
\begin{matrix}
    (a_i+\sigma^{-1}A_i,\sigma^{-1}A_i,\alpha_i)_{1,p}&(1,1,1)\\ 
    (0,1,1)&(b_j+\sigma^{-1}B_j,\sigma^{-1}B_j,\beta_j)_{1,q}
  \end{matrix}
  \right.\Bigg].
\end{equation*}
The hazard rate function $\lambda_X(x)$ is given by
\begin{eqnarray*}
\lambda_X(x)&=&\frac{f_X(x)}{\overline{F}_X(x)}\\
&=&\frac{\sigma\delta^{\sigma^{-1}}I^{m,n}_{p,q}\Bigg[\delta x^{\sigma}\left|
\begin{matrix}
    (a_i,A_i,\alpha_i)_{1,p}\\ 
    (b_j,B_j,\beta_j)_{1,q}
  \end{matrix}
\right.\Bigg]
}{
I^{m+1,n}_{p+1,q+1}\Bigg[\delta^{\sigma^{-1}} x\left|
\begin{matrix}
    (a_i+\sigma^{-1}A_i,\sigma^{-1}A_i,\alpha_i)_{1,p}&(1,1,1)\\ 
    (0,1,1)&(b_j+\sigma^{-1}B_j,\sigma^{-1}B_j,\beta_j)_{1,q}
  \end{matrix}
  \right.\Bigg]}.
\end{eqnarray*}
\subsubsection*{(d) Order statistics}
Let $X_{(1)}\leq X_{(2)}\leq\ldots\leq X_{(N)}$ denotes the order statistics of a random sample $X_{1},X_{2},\ldots, X_{N}$ of size $N$ drawn from the $I$-function distribution. The pdf of the $j$-th order statistics is given by
\begin{equation*}
f_{X_{(j)}}(x)=\frac{N!}{(j-1)!(N-j)!}f_{X}(x)\left\{F_{X}(x)\right\}^{j-1}\left\{\overline{F}_{X}(x)\right\}^{N-j}.
\end{equation*}
Therefore, the pdf of the minimum and the maximum order statistics is given by
\begin{eqnarray*}
f_{X_{(1)}}(x)&=&\frac{Nk_0^N}{\sigma^{N-1}\delta^{(N-1)\sigma^{-1}}}I^{m,n}_{p,q}\Bigg[\delta x^{\sigma}\left|
\begin{matrix}
    (a_i,A_i,\alpha_i)_{1,p}\\ 
    (b_j,B_j,\beta_j)_{1,q}
  \end{matrix}
\right.\Bigg]
\\&&\left\{I^{m+1,n}_{p+1,q+1}\Bigg[\delta^{\sigma^{-1}} x\left|
\begin{matrix}
    (a_i+\sigma^{-1}A_i,\sigma^{-1}A_i,\alpha_i)_{1,p}&(1,1,1)\\ 
    (0,1,1)&(b_j+\sigma^{-1}B_j,\sigma^{-1}B_j,\beta_j)_{1,q}
  \end{matrix}
  \right.\Bigg]\right\}^{N-1}
\end{eqnarray*}
and
\begin{eqnarray*}
f_{X_{(N)}}(x)&=&Nk_0I^{m,n}_{p,q}\Bigg[\delta x^{\sigma}\left|
\begin{matrix}
    (a_i,A_i,\alpha_i)_{1,p}\\ 
    (b_j,B_j,\beta_j)_{1,q}
  \end{matrix}
\right.\Bigg]
\Bigg\{1-\frac{k_0}{\sigma\delta^{\sigma^{-1}}}
\\&&I^{m+1,n}_{p+1,q+1}\Bigg[\delta^{\sigma^{-1}} x\Bigg|\left.
\begin{matrix}
    (a_i+\sigma^{-1}A_i,\sigma^{-1}A_i,\alpha_i)_{1,p}&(1,1,1)\\ 
    (0,1,1)&(b_j+\sigma^{-1}B_j,\sigma^{-1}B_j,\beta_j)_{1,q}
  \end{matrix}
  \Bigg]\right\}^{N-1}.
\end{eqnarray*}
\subsection{Some special cases} 
The density of the $H$-function distribution (see \cite{Springer542}),
\begin{equation}
f_X(x) =\left\{
	\begin{array}{ll}
	   \frac{\delta}{\chi(1)}H^{m,n}_{p,q}\Bigg[\delta x\left|
\begin{matrix}
    (a_i,A_i)_{1,p}\\ 
    (b_j,B_j)_{1,q}
  \end{matrix}
\right.\Bigg], & \mbox{} x>0,\\\\
		0  & \mbox{} \mathrm{otherwise},
	\end{array}
\right.
\end{equation}
follows as special case of the $I$-function distribution (\ref{4.1}) when $\sigma=1$, $\alpha_i$'s and $\beta_j$'s for all $i$ and $j$ equals unity. Hence, the other standard non-negative distributions are particular cases. We next give the densities of the well-known distributions in terms of the $I$-function. Let $X$ be a rv which follows\\
(i) the gamma distribution with parameters $\theta,\lambda>0$. Then
\begin{equation}\label{n3.8}
f_X(x)=\frac{1}{\lambda^{\theta}\Gamma(\theta)}x^{\theta-1}e^{-\frac{x}{\lambda}}
=\frac{1}{\lambda\Gamma(\theta)}I^{1,0}_{0,1}\Bigg[\frac{x}{\lambda}\left|
\begin{matrix}
    \\ 
    (\theta-1,1,1)
  \end{matrix}
\right.\Bigg],\ \ x>0.
\end{equation}
(ii) the exponential distribution (put $\theta=1$ in (\ref{n3.8})) with parameter $\lambda>0$. Then
\begin{equation}\label{m3.9}
f_X(x)=\frac{1}{\lambda}e^{-\frac{x}{\lambda}}
=\frac{1}{\lambda}I^{1,0}_{0,1}\Bigg[\frac{x}{\lambda}\left|
\begin{matrix}
    \\ 
    (0,1,1)
  \end{matrix}
\right.\Bigg],\ \ x>0.
\end{equation}
(iii) the chi-square distribution (put $\theta=\frac{\nu}{2}$ and $\lambda=2$ in (\ref{n3.8})) with parameter $\nu\in\mathbb{N}$ denotes the degree of freedom. Then
\begin{equation*}
f_X(x)=\frac{1}{2^{\frac{\nu}{2}}\Gamma\left(\frac{\nu}{2}\right)}x^{\frac{\nu}{2}-1}e^{-\frac{x}{2}}
=\frac{1}{2\Gamma\left(\frac{\nu}{2}\right)}I^{1,0}_{0,1}\Bigg[\frac{x}{2}\left|
\begin{matrix}
    \\ 
    \left(\frac{\nu}{2}-1,1,1\right)
  \end{matrix}
\right.\Bigg],\ \ x>0.
\end{equation*}
(iv) the Weibull distribution with parameters $\theta,\lambda>0$. Then
\begin{equation}\label{n3.9}
f_X(x)=\frac{\theta}{\lambda} x^{\theta-1}e^{-\frac{x^{\theta}}{\lambda}}
=\frac{1}{\lambda^{\frac{1}{\theta}}}I^{1,0}_{0,1}\Bigg[\frac{x}{\lambda^{\frac{1}{\theta}}}\left|
\begin{matrix}
    \\ 
    \left(1-\frac{1}{\theta},\frac{1}{\theta},1\right)
  \end{matrix}
\right.\Bigg],\ \  x>0.
\end{equation}
(v) the Rayleigh distribution (put $\theta=2$ and $\lambda=2\nu^2$ in (\ref{n3.9})) with parameter $\nu>0$. Then
\begin{equation*}
f_X(x)=\frac{1}{\nu^2} xe^{-\frac{x^2}{2\nu^2}}
=\frac{1}{\nu\sqrt{2}}I^{1,0}_{0,1}\Bigg[\frac{x}{\nu\sqrt{2}}\left|
\begin{matrix}
    \\ 
    \left(\frac{1}{2},\frac{1}{2},1\right)
  \end{matrix}
\right.\Bigg],\ \ x>0.
\end{equation*}
(vi) the Maxwell distribution with parameter $\lambda>0$. Then
\begin{equation*}
f_X(x)=\frac{4}{\lambda^3\sqrt{\pi}}x^{2}e^{-\frac{x^2}{\lambda^2}}
=\frac{2}{\lambda\sqrt{\pi}}I^{1,0}_{0,1}\Bigg[\frac{x}{\lambda}\left|
\begin{matrix}
    \\ 
    \left(1,\frac{1}{2},1\right)
  \end{matrix}
\right.\Bigg],\ \ x>0.
\end{equation*}
(vii) the half-normal distribution with parameter $\lambda>0$. Then
\begin{equation}\label{m3.11}
f_X(x)=\frac{\sqrt{2}}{\lambda\sqrt{\pi}}e^{-\frac{x^2}{2\lambda^2}}
=\frac{1}{\lambda\sqrt{2\pi}}I^{1,0}_{0,1}\Bigg[\frac{x}{\lambda\sqrt{2}}\left|
\begin{matrix}
    \\ 
    \left(0,\frac{1}{2},1\right)
  \end{matrix}
\right.\Bigg],\ \ x>0.
\end{equation}
(viii) the half-Cauchy distribution with parameter $\lambda>0$. Then
\begin{equation}\label{m3.12}
f_X(x)=\frac{2\lambda}{\pi}\frac{1}{(\lambda^2+x^2)}
=\frac{1}{\lambda\pi}I^{1,1}_{1,1}\Bigg[\frac{x}{\lambda}\left|
\begin{matrix}
    \left(0,\frac{1}{2},1\right)\\
    \left(0,\frac{1}{2},1\right)
  \end{matrix}
\right.\Bigg],\ \ x>0.
\end{equation}
(ix) the beta distribution of first kind with parameters $\theta,\lambda>0$. Then
\begin{equation*}
f_X(x)=\frac{\Gamma(\theta+\lambda)}{\Gamma(\theta)\Gamma(\lambda)}x^{\theta-1}(1-x)^{\lambda-1}
=\frac{\Gamma(\theta+\lambda)}{\Gamma(\theta)}I^{1,0}_{1,1}\Bigg[x\left|
\begin{matrix}
    (\theta+\lambda-1,1,1)\\ 
    (\theta-1,1,1)
  \end{matrix}
\right.\Bigg],\ \ 0<x<1.
\end{equation*}
(x) the beta distribution of second kind with parameters $\theta,\lambda>0$. Then
\begin{equation*}
f_X(x)=\frac{\Gamma(\theta+\lambda)}{\Gamma(\theta)\Gamma(\lambda)}\left(\frac{\lambda}{\theta}\right)^{\theta}x^{\theta-1}\left(1+\frac{\lambda x}{\theta}\right)^{-(\theta+\lambda)}
=\frac{\lambda}{\theta\Gamma(\theta)\Gamma(\lambda)}I^{1,1}_{1,1}\Bigg[\frac{\lambda x}{\theta}\left|
\begin{matrix}
    (-\lambda,1,1)\\ 
    (\theta-1,1,1)
  \end{matrix}
\right.\Bigg],\ \ x>0.
\end{equation*}
(xi) the power function distribution with parameter $\theta>0$. Then
\begin{equation}\label{n3.10}
f_X(x)=\theta x^{\theta-1}
=\theta I^{1,0}_{1,1}\Bigg[x\left|
\begin{matrix}
    (\theta,1,1)\\ 
    (\theta-1,1,1)
  \end{matrix}
\right.\Bigg],\ \ 0<x<1.
\end{equation}
(xii) the uniform distribution on $(0,1)$ (put $\theta=1$ in (\ref{n3.10})). Then
\begin{equation}\label{u3.11}
f_X(x)=1
=I^{1,0}_{1,1}\Bigg[x\left|
\begin{matrix}
    (1,1,1)\\ 
    (0,1,1)
  \end{matrix}
\right.\Bigg],\ \ 0<x<1.
\end{equation}
(xiii) the Pareto distribution with parameter $\lambda>0$. Then
\begin{equation*}
f_X(x)=\lambda x^{-(\lambda+1)}
=\lambda I^{0,1}_{1,1}\Bigg[x\left|
\begin{matrix}
    (-\lambda,1,1)\\ 
    (-\lambda-1,1,1)
  \end{matrix}
\right.\Bigg],\ \ x>1.
\end{equation*}
(xiv) the half-student distribution with parameter $\nu>0$. Then
\begin{equation*}
f_X(x)=\frac{2\Gamma\left(\frac{\nu+1}{2}\right)}{\sqrt{\nu\pi}\Gamma\left(\frac{\nu}{2}\right)}\left(1+\frac{x^2}{\nu}\right)^{-\frac{\nu+1}{2}}
=\frac{1}{\sqrt{\nu\pi}\Gamma\left(\frac{\nu}{2}\right)} I^{1,1}_{1,1}\Bigg[\frac{x}{\sqrt{\nu}}\left|
\begin{matrix}
    \left(\frac{1-\nu}{2},\frac{1}{2},1\right)\\ 
    \left(0,\frac{1}{2},1\right)
  \end{matrix}
\right.\Bigg],\ \ x>0.
\end{equation*}
(xv) the $F$ distribution with parameters $\theta,\lambda>0$. Then
\begin{equation*}
f_X(x)=\frac{\Gamma\left(\frac{\theta+\lambda}{2}\right)\theta^{\frac{\theta}{2}}\lambda^{\frac{\lambda}{2}}}{\Gamma\left(\frac{\theta}{2}\right)\Gamma\left(\frac{\lambda}{2}\right)}x^{\frac{\theta}{2}-1}\left(\lambda+\theta x\right)^{-\frac{(\theta+\lambda)}{2}}
=\frac{\theta}{\lambda\Gamma\left(\frac{\theta}{2}\right)\Gamma\left(\frac{\lambda}{2}\right)}I^{1,1}_{1,1}\Bigg[\frac{\theta x}{\lambda}\left|
\begin{matrix}
    \left(-\frac{\lambda}{2},1,1\right)\\ 
    \left(\frac{\theta}{2}-1,1,1\right)
  \end{matrix}
\right.\Bigg],\ \ x>0.
\end{equation*}
(xvi) the general hypergeometric distribution \cite{Mathai127}. Then
\begin{eqnarray*}
f_X(x)&=&\frac{da^{\frac{c}{d}}\Gamma(\alpha)\Gamma\left(\beta-\frac{c}{d}\right)}{\Gamma\left(\frac{c}{d}\right)\Gamma(\beta)\Gamma\left(\alpha-\frac{c}{d}\right)}x^{c-1}{}_1F_1\left(\alpha;\beta;-ax^d\right)\\
&=&\frac{a^{\frac{1}{d}}\Gamma\left(\beta-\frac{c}{d}\right)}{\Gamma\left(\frac{c}{d}\right)\Gamma\left(\alpha-\frac{c}{d}\right)}I^{1,1}_{1,2}\Bigg[a^{\frac{1}{d}}x\left|
\begin{matrix}
    \left(1-\alpha+\frac{c-1}{d},\frac{1}{d},1\right)\\ 
    \left(\frac{c-1}{d},\frac{1}{d},1\right)&\left(1-\beta+\frac{c-1}{d},\frac{1}{d},1\right)
  \end{matrix}
\right.\Bigg],\ \ x>0.
\end{eqnarray*}
The above relationships can easily be obtained using identities (\ref{m2.8})-(\ref{m2.10}) and properties of the $I$-function (\ref{2.8})-(\ref{2.9}).
\subsection{Properties of the $I$-function distribution}
\setcounter{equation}{0}
\noindent Next we show that the class of $I$-function variates is closed under multiplication and quotients of independent $I$-function variates. Moreover, this class is also closed under positive scalar multiplication and rational power of the $I$-function variate.
\begin{theorem}\label{t4.1}
Let $X_1,X_2,\ldots,X_N$ be $N$ independent $I$-function rv's with pdf's $f_{X_1}(x_1),$ $f_{X_2}(x_2),\ldots,f_{X_N}(x_N)$, respectively, where
\begin{equation}\label{r4.1}
f_{X_l}(x_l)=\frac{\sigma_l\delta_l^{\sigma_l^{-1}}}{\chi_l(\sigma_l^{-1})}I^{m_l,n_l}_{p_l,q_l}\Bigg[\delta_l x_l^{\sigma_l}\left|
\begin{matrix}
    (a_{li},A_{li},\alpha_{li})_{1,p_l}\\ 
    (b_{lj},B_{lj},\beta_{lj})_{1,q_l}
  \end{matrix}
\right.\Bigg],\ \ x_l>0,
\end{equation}
and
\begin{equation*}
\chi_l(s)=\frac{\prod_{i=1}^{n_l}\Gamma^{\alpha_{li}}\left(1-a_{li}-A_{li}s\right)\prod_{j=1}^{m_l}\Gamma^{\beta_{lj}}\left(b_{lj}+B_{lj}s\right)}{\prod_{i=n_l+1}^{p_l}\Gamma^{\alpha_{li}}\left(a_{li}+A_{li}s\right)\prod_{j=m_l+1}^{q_l}\Gamma^{\beta_{lj}}\left(1-b_{lj}-B_{lj}s\right)},
\end{equation*}
for $l=1,2,\ldots,N$. Then, the pdf of the rv $Y=\prod_{l=1}^NX_l$ is given by
\begin{equation*}
f_{Y}(y)=kI^{m,n}_{p,q}\Bigg[\delta y\left|
\begin{matrix}
    (a_{i},A_{i},\alpha_{i})_{1,p}\\ 
    (b_{j},B_{j},\beta_{j})_{1,q}
  \end{matrix}
\right.\Bigg],\ \ y>0,
\end{equation*}
where $k=\prod_{l=1}^N\frac{\delta_l^{\sigma_l^{-1}}}{\chi_l(\sigma_l^{-1})}$, $m=\sum_{l=1}^Nm_l$, $n=\sum_{l=1}^Nn_l$, $p=\sum_{l=1}^Np_l$, $q=\sum_{l=1}^Nq_l$, $\delta=\prod_{l=1}^N\delta_l^{\sigma_l^{-1}}$ and the sequence of parameters are as follows:
\begin{eqnarray*}
&&(a_{i},A_{i},\alpha_{i})_{1,p}=(a_{1i},\sigma_l^{-1}A_{1i},\alpha_{1i})_{1,n_1}(a_{2i},\sigma_l^{-1}A_{2i},\alpha_{2i})_{1,n_2}\ldots(a_{Ni},\sigma_l^{-1}A_{Ni},\alpha_{Ni})_{1,n_N}\\&&(a_{1i},\sigma_l^{-1}A_{1i},\alpha_{1i})_{n_1+1,p_1}(a_{2i},\sigma_l^{-1}A_{2i},\alpha_{2i})_{n_2+1,p_2}\ldots(a_{Ni},\sigma_l^{-1}A_{Ni},\alpha_{Ni})_{n_N+1,p_N},\\
&&(b_{j},B_{j},\beta_{j})_{1,q}=(b_{1j},\sigma_l^{-1}B_{1j},\beta_{1j})_{1,m_1}(b_{2j},\sigma_l^{-1}B_{2j},\beta_{2j})_{1,m_2}\ldots(b_{Nj},\sigma_l^{-1}B_{Nj},\beta_{Nj})_{1,m_N}\\&&(b_{1j},\sigma_l^{-1}B_{1j},\beta_{1j})_{m_1+1,q_1}(b_{2j},\sigma_l^{-1}B_{2j},\beta_{2j})_{m_2+1,q_2}\ldots(b_{Nj},\sigma_l^{-1}B_{Nj},\beta_{Nj})_{m_N+1,q_N}.
\end{eqnarray*}
\end{theorem}
\begin{proof}
From (\ref{2.6}), we have for $y>0$
\begin{equation*}
\mathcal{M}_{f_{Y}}(s)=\prod_{l=1}^{N}\mathcal{M}_{f_{X_l}}(s)=\prod_{l=1}^{N}\frac{\delta_l^{\sigma_l^{-1}}\chi_l(\sigma_l^{-1}s)}{\delta_l^{\sigma_l^{-1}s}\chi_l(\sigma_l^{-1})}.
\end{equation*}
Now using the inverse Mellin transform (\ref{2.2}), the pdf is given by
\begin{eqnarray*}
f_Y(y)&=&\left(\prod_{l=1}^{N}\frac{\delta_l^{\sigma_l^{-1}}}{\chi_l(\sigma_l^{-1})}\right)\frac{1}{2\pi i}\int_{c-i\infty}^{c+i\infty}\prod_{l=1}^{N}\chi_l(\sigma_l^{-1}s)\left(y\prod_{l=1}^{N}\delta_l^{\sigma_l^{-1}}\right)^{-s}\,ds,
\end{eqnarray*}
which on using (\ref{1.2}) gives the result.
\end{proof}

\begin{example}\label{e4.1}
\noindent Let $X_1,X_2,\ldots,X_n$ be $n$ independent rv's such that $X_i$ follows half-normal distribution for $i=1,2,\ldots,m$ and beta distribution for $i=m+1,\ldots,n$, with pdf's
\begin{eqnarray*}
f_{X_i}(x_i) &=& \frac{\sqrt{2}}{\sigma_i\sqrt{\pi}}e^{-\frac{x_i^2}{2\sigma_i^2}},\ \ \ x_i>0,\ \ i=1,2,\ldots,m,\\
\mathrm{and}\ \ \ \ \ \ f_{X_i}(x_i) &=& \frac{\Gamma(a_i+b_i)}{\Gamma(a_i)\Gamma(b_i)}x_i^{a_i-1}(1-x_i)^{b_i-1},\ \ \ 0<x_i<1,\ \ i=m+1,\ldots,n,
\end{eqnarray*}
where the parameters $\sigma_i,a_i,b_i>0$. The corresponding Mellin transforms (\ref{2.1}) are given by
\begin{equation*}
\mathcal{M}_{f_{X_i}}(s) =\left\{
	\begin{array}{ll}
	    \frac{1}{\sqrt{2\pi}\sigma_i}\left(\frac{1}{\sqrt{2}\sigma_i}\right)^{-s}\Gamma\left(\frac{s}{2}\right), & \mbox{} i=1,2,\ldots,m,\\\\
		\frac{\Gamma(a_i+b_i)\Gamma(a_i-1+s)}{\Gamma(a_i)\Gamma(a_i+b_i-1+s)},  & \mbox{} i=m+1,\ldots,n.
	\end{array}
\right.
\end{equation*}
Let $Y=\prod_{i=1}^nX_i$. Using (\ref{2.6}), its Mellin transform is
\begin{equation}\label{m4.5}
\mathcal{M}_{f_{Y}}(s)=\Gamma^m\left(\frac{s}{2}\right)\prod_{i=1}^m\frac{(\sqrt{2}\sigma_i)^s}{\sqrt{2\pi}\sigma_i}\prod_{i=m+1}^n\frac{\Gamma(a_i+b_i)\Gamma(a_i-1+s)}{\Gamma(a_i)\Gamma(a_i+b_i-1+s)}.
\end{equation}
Using (\ref{2.2}) and (\ref{m4.5}), the pdf of rv $Y$ is
\begin{eqnarray}\label{1.1}
f_{Y}(y) &=&\frac{k}{2\pi i}\int_{c-i\infty}^{c+i\infty}\Gamma^m\left(\frac{s}{2}\right)\prod_{i=m+1}^n\frac{\Gamma(a_i-1+s)}{\Gamma(a_i+b_i-1+s)}\left(y\prod_{i=1}^m\frac{1}{\sqrt{2}\sigma_i}\right)^{-s}\,ds\\
&=&kI^{n-m+1,0}_{n-m,n-m+1}\Bigg[\frac{2^{-\frac{m}{2}}y}{\prod_{i=1}^m\sigma_i}\left|
\begin{matrix}
    (a_i+b_i-1,1,1)_{m+1,n}\\ 
    \left(0,\frac{1}{2},m\right)&(a_i-1,1,1)_{m+1,n}
  \end{matrix}
\right.\Bigg],\ y>0,\nonumber
\end{eqnarray}
where $k=\prod_{i=1}^m\frac{1}{\sqrt{2\pi}\sigma_i}\prod_{i=m+1}^n\frac{\Gamma(a_i+b_i)}{\Gamma(a_i)} $. Note that the integral (\ref{1.1}) can also be represented in terms of the Fox's $H$-function as
\begin{equation*}
f_{Y}(y) = \frac{\prod_{i=m+1}^n\frac{\Gamma(a_i+b_i)}{\Gamma(a_i)}}{(2\pi)^{\frac{m}{2}}\prod_{i=1}^m\sigma_i}H^{n,0}_{n-m,n}\Bigg[\frac{2^{-\frac{m}{2}}y}{\prod_{i=1}^m\sigma_i}\left|
\begin{matrix}
    (a_i+b_i-1,1)_{m+1,n}\\ 
    (0,\frac{1}{2})_{1,m}&(a_i-1,1)_{m+1,n}
  \end{matrix}
\right.\Bigg],\ y>0.
\end{equation*}
\end{example}

\begin{example}
Let $X_1,X_2$ be two independent rv's such that $X_1$ is uniformly distributed $(\ref{u3.11})$ on $(0,1)$ and $X_2$ follows gamma distribution (\ref{n3.8}) with shape parameter $\theta=2$. Using Theorem \ref{t4.1}, the pdf of rv $Y=X_1X_2$ is
\begin{equation*}
f_Y(y)=\frac{1}{\lambda}I^{1,0}_{0,1}\Bigg[\frac{y}{\lambda}\left|
\begin{matrix}
    \\
    \left(0,1,1\right)
  \end{matrix}
\right.\Bigg],\ \ y>0,
\end{equation*}
which is an exponential distribution with parameter $\lambda$ (see \cite{Johnson1995}, p. 306).
\end{example}

\begin{theorem}
Let $X$ be an $I$-function rv with pdf (\ref{4.1}).\\
\noindent (a) The pdf of the rv $Y=aX,\ a>0$, is given by
\begin{equation}\label{n4.3}
f_{Y}(y)=\frac{\delta^{\sigma^{-1}}}{a\chi(\sigma^{-1})}I^{m,n}_{p,q}\Bigg[\frac{\delta^{\sigma^{-1}}y}{a}\left|
\begin{matrix}
     (a_i,\sigma^{-1}A_i,\alpha_i)_{1,p}\\ 
    (b_j,\sigma^{-1}B_j,\beta_j)_{1,q}
  \end{matrix}
\right.\Bigg],\ \ y>0.
\end{equation}
\noindent (b) The pdf of the rv $Y=X^r$ for $r\neq 0$ rational is given by
\begin{equation}\label{4.3}
f_{Y}(y)=\frac{\delta^{r\sigma^{-1}}}{\chi(\sigma^{-1})}I^{m,n}_{p,q}\Bigg[\delta^{r\sigma^{-1}} y\left|
\begin{matrix}
    (a_i-r\sigma^{-1}A_i+\sigma^{-1}A_i,r\sigma^{-1}A_i,\alpha_i)_{1,p}\\ 
    (b_j-r\sigma^{-1}B_j+\sigma^{-1}B_j,r\sigma^{-1}B_j,\beta_j)_{1,q}
  \end{matrix}
\right.\Bigg],\ \  y>0,
\end{equation}
when $r>0$ and 
\begin{equation}\label{4.4}
f_{Y}(y)=\frac{\delta^{r\sigma^{-1}}}{\chi(\sigma^{-1})}I^{n,m}_{q,p}\Bigg[\delta^{r\sigma^{-1}} y\left|
\begin{matrix}
    (1-b_j+r\sigma^{-1}B_j-\sigma^{-1}B_j,-r\sigma^{-1}B_j,\beta_j)_{1,q}\\
    (1-a_i+r\sigma^{-1}A_i-\sigma^{-1}A_i,-r\sigma^{-1}A_i,\alpha_i)_{1,p}
      \end{matrix}
\right.\Bigg],\ \ y>0,
\end{equation}
when $r<0$.
\end{theorem}
\begin{proof} (a) From (\ref{2.3}), we have
\begin{equation*}
\mathcal{M}_{f_{Y}}(s)=a^{s-1}\mathcal{M}_{f_{X}}(s)=\frac{a^{s-1}\delta^{\sigma^{-1}}\chi(\sigma^{-1}s)}{\delta^{\sigma^{-1}s}\chi(\sigma^{-1})}.
\end{equation*}
Using (\ref{2.2}), the pdf of $Y$ is 
\begin{eqnarray*}
f_Y(y)&=&\frac{\delta^{\sigma^{-1}}}{a\chi(\sigma^{-1})}\frac{1}{2\pi i}\int_{c-i\infty}^{c+i\infty}\chi(\sigma^{-1}s)\left(\frac{\delta^{\sigma^{-1}}y}{a}\right)^{-s}\,ds,\ \ \ \ y>0,
\end{eqnarray*}
which gives (\ref{n4.3}) on using (\ref{1.2}).\\
\noindent (b) From (\ref{2.4}), we have
\begin{equation*}
\mathcal{M}_{f_{Y}}(s)=\mathcal{M}_{f_{X}}(rs-r+1)=\frac{\delta^{r\sigma^{-1}}\chi(\sigma^{-1}rs-\sigma^{-1}r+\sigma^{-1})}{\delta^{r\sigma^{-1}s}\chi(\sigma^{-1})}.
\end{equation*}
Therefore,
\begin{equation}\label{4.5}
f_Y(y)=\frac{\delta^{r\sigma^{-1}}}{\chi(\sigma^{-1})}\frac{1}{2\pi i}\int_{c-i\infty}^{c+i\infty}\chi(\sigma^{-1}rs-\sigma^{-1}r+\sigma^{-1}) \left(y\delta^{r\sigma^{-1}}\right)^{-s}\,ds.
\end{equation}
When $r>0$, the lhs of (\ref{4.5}) on using (\ref{1.2}) gives (\ref{4.3}). When $r<0$, put $t=-r$ so that for $t>0$,
\begin{equation*}
f_Y(y)=\frac{\delta^{-t\sigma^{-1}}}{\chi(\sigma^{-1})}\frac{1}{2\pi i}\int_{c-i\infty}^{c+i\infty}\chi(-\sigma^{-1}ts+\sigma^{-1}t+\sigma^{-1}) \left(y\delta^{-t\sigma^{-1}}\right)^{-s}\,ds,\ \ \ \ y>0,
\end{equation*}
which on using (\ref{1.2}) gives (\ref{4.4}).
\end{proof}

\begin{example}
Let $Y$ be a rv defined in Example \ref{e4.1}. Consider $W=4Y$ and $Z=Y^{-1}$. Using (\ref{n4.3}) the pdf of rv $W$ is
\begin{equation*}
f_{W}(w)=\frac{k}{4}I^{n-m+1,0}_{n-m,n-m+1}\Bigg[\frac{2^{-\frac{m}{2}}w}{4\prod_{i=1}^m\sigma_i}\left|
\begin{matrix}
    (a_i+b_i-1,1,1)_{m+1,n}\\ 
    \left(0,\frac{1}{2},m\right)&(a_i-1,1,1)_{m+1,n}
  \end{matrix}
\right.\Bigg],\ w>0,
\end{equation*}
where $k=\prod_{i=1}^m\frac{1}{\sqrt{2\pi}\sigma_i}\prod_{i=m+1}^n\frac{\Gamma(a_i+b_i)}{\Gamma(a_i)} $.\\ Similarly by using (\ref{4.4}) the pdf of rv $Z$ is 
\begin{equation*}
f_{Z}(z)=k^*I^{0,n-m+1}_{n-m+1,n-m}\Bigg[z\prod_{i=1}^m\sqrt{2}\sigma_i\left|
\begin{matrix}
    \left(0,\frac{1}{2},m\right)&(-a_i,1,1)_{m+1,n}\\ 
    (-a_i-b_i,1,1)_{m+1,n}
  \end{matrix}
\right.\Bigg],\ z>0,
\end{equation*}
where $k^*=\prod_{i=1}^m\sqrt{\frac{2}{\pi}}\sigma_i\prod_{i=m+1}^n\frac{\Gamma(a_i+b_i)}{\Gamma(a_i)} $. 
\end{example}

\begin{example}\label{e4.2}
Let $X$ be an exponential rv with pdf given by (\ref{m3.9}). Consider $Y=\theta X$ and $Z=X^{\frac{1}{\theta}}$ where $\theta>0$. Using (\ref{n4.3}) the pdf of rv $Y$ is
\begin{equation*}
f_Y(y)=\frac{1}{\beta}I^{1,0}_{0,1}\Bigg[\frac{y}{\beta}\left|
\begin{matrix}
    \\ 
    (0,1,1)
  \end{matrix}
\right.\Bigg],\ \ y>0,
\end{equation*}
where $\beta=\lambda\theta$. Now by using (\ref{4.3}) the pdf of rv $Z$ is 
\begin{equation*}
f_Z(z)=\frac{1}{\lambda^{\frac{1}{\theta}}}I^{1,0}_{0,1}\Bigg[\frac{z}{\lambda^{\frac{1}{\theta}}}\left|
\begin{matrix}
    \\ 
    \left(1-\frac{1}{\theta},\frac{1}{\theta},1\right)
  \end{matrix}
\right.\Bigg],\ \  z>0.
\end{equation*}
Thus, $Y$ is exponentially distributed with parameter $\beta$, whereas $Z$ follows the Weibull distribution (\ref{n3.9}) with parameters $\lambda,\theta$ (see \cite{Leemis45}).
\end{example}

\begin{theorem}\label{t4.3}
Let $X_1$and $X_2$ be two independent $I$-function rv's with pdf's given by (\ref{r4.1}) for $l=1,2$. Then, the pdf of the rv $Y=\frac{X_1}{X_2}$ is given by
\begin{equation*}
f_{Y}(y)=kI^{m,n}_{p,q}\Bigg[\delta y\left|
\begin{matrix}
    (a_{i},A_{i},\alpha_{i})_{1,p}\\ 
    (b_{j},B_{j},\beta_{j})_{1,q}
  \end{matrix}
\right.\Bigg],\ \ y>0,
\end{equation*}
where $k=\frac{\delta_1^{\sigma_1^{-1}}\delta_2^{-\sigma_2^{-1}}}{\chi_1(\sigma_1^{-1})\chi_2(\sigma_2^{-1})}$, $m=m_1+n_2$, $n=n_1+m_2$, $p=p_1+q_2$, $q=q_1+p_2$, $\delta=\delta_1^{\sigma_1^{-1}}\delta_2^{-\sigma_2^{-1}}$ and the sequence of parameters are
\begin{equation*}
(a_{i},A_{i},\alpha_{i})_{1,p}=(a_{1i},\sigma_1^{-1}A_{1i},\alpha_{1i})_{1,n_1}(1-b_{2j}-2\sigma_2^{-1}B_{2j},\sigma_2^{-1}B_{2j},\beta_{2j})_{1,q_2}(a_{1i},A_{1i},\sigma_1^{-1}\alpha_{1i})_{n_1+1,p_1},
\end{equation*}
\begin{equation*}
(b_{j},B_{j},\beta_{j})_{1,q}=(b_{1j},\sigma_1^{-1}B_{1j},\beta_{1j})_{1,m_1}(1-a_{2i}-2\sigma_2^{-1}A_{2i},\sigma_2^{-1}A_{2i},\alpha_{2i})_{1,p_2}(b_{1j},\sigma_1^{-1}B_{1j},\beta_{1j})_{m_1+1,q_1}.
\end{equation*}
\end{theorem}
\begin{proof}
From (\ref{2.5}), we have for $y>0$,
\begin{eqnarray*}
f_Y(y) &=&\mathcal{M}^{-1}\left[\mathcal{M}_{f_{X_1}}(s)\mathcal{M}_{f_{X_2}}(2-s)\right]\\
&=&\frac{\delta_1^{\sigma_1^{-1}}\delta_2^{-\sigma_2^{-1}}}{\chi_1(\sigma_1^{-1})\chi_2(\sigma_2^{-1})}\frac{1}{2\pi i}\int_{c-i\infty}^{c+i\infty}\chi_1(\sigma_1^{-1}s)\chi_2(2\sigma_2^{-1}-\sigma_2^{-1}s)\left(y\delta_1^{\sigma_1^{-1}}\delta_2^{-\sigma_2^{-1}}\right)^{-s}\,ds,
\end{eqnarray*}
which on using the definition of the $I$-function (\ref{1.2}) completes the proof.
\end{proof}

\begin{example}
Let $X_1,X_2$ be two independent half-normal rv's (\ref{m3.11}) with pdf's 
\begin{equation*}
f_{X_i}(x_i)=\frac{1}{\lambda_i\sqrt{2\pi}}I^{1,0}_{0,1}\Bigg[\frac{x_i}{\lambda_i\sqrt{2}}\left|
\begin{matrix}
    \\ 
    \left(0,\frac{1}{2},1\right)
  \end{matrix}
\right.\Bigg],\ \ x_i>0,
\end{equation*}
where the parameters $\lambda_i>0$ for $i=1,2$. Using Theorem \ref{t4.3}, the rv $Y=X_1/X_2$ follows Cauchy distribution (\ref{m3.12}) with pdf
\begin{equation*}
f_Y(y)=\frac{1}{\lambda\pi}I^{1,1}_{1,1}\Bigg[\frac{y}{\lambda}\left|
\begin{matrix}
    \left(0,\frac{1}{2},1\right)\\
    \left(0,\frac{1}{2},1\right)
  \end{matrix}
\right.\Bigg],\ \ y>0,
\end{equation*}
where the parameter $\lambda=\lambda_1/\lambda_2$ (see \cite{Leemis45}).
\end{example}
\subsection{The distribution based on the product of two $I$-functions}
Let $X$ be a rv with pdf
\begin{equation}\label{nn4.1}
f_X(x) =\left\{
	\begin{array}{ll}
	   k_1I^{m,n}_{p,q}\Bigg[\delta x^{\sigma}\left|
\begin{matrix}
    (a_i,A_i,\alpha_i)_{1,p}\\ 
    (b_j,B_j,\beta_j)_{1,q}
  \end{matrix}
\right.\Bigg]
I^{k,l}_{u,v}\Bigg[\eta x^\mu\left|
\begin{matrix}
    (c_i,C_i,\gamma_i)_{1,u}\\ 
    (d_j,D_j,\rho_j)_{1,v}
  \end{matrix}
\right.\Bigg], & \mbox{} x>0,\\\\
		0  & \mbox{} \mathrm{otherwise},
	\end{array}
\right.
\end{equation}
where $\sigma>0,\mu\geq0$ and $\delta,\eta\neq0$. Also, $(a_i,A_i,\alpha_i)_{1,p},(b_j,B_j,\beta_j)_{1,q},(c_i,C_i,\gamma_i)_{1,u},(d_j,D_j,\rho_j)_{1,v},$ $m,n,k,l$ must conform to those restrictions in the definition of the $I$-function (\ref{1.2}), such that $f_X(x)\geq 0$. Using Lemma \ref{l3.1}, the normalizing constant $k_1$ is
\begin{equation*}
k_1^{-1}=\frac{\sigma^{-1}}{\delta^{\sigma^{-1}}}I^{n+k,m+l}_{q+u,p+v}\Bigg[\frac{\eta}{\delta^{\sigma^{-1}\mu}}\left|
\begin{matrix}
    (c_i,C_i,\gamma_i)_{1,l}&(1-b_j-\sigma^{-1}B_j,\mu\sigma^{-1} B_j,\beta_j)_{1,q}&(c_i,C_i,\gamma_i)_{l+1,u}\\ 
    (d_j,D_j,\rho_j)_{1,k}&(1-a_i-\sigma^{-1}A_i,\mu\sigma^{-1} A_i,\alpha_i)_{1,p}&(d_j,D_j,\rho_j)_{k+1,v}
  \end{matrix}
\right.\Bigg].
\end{equation*}
\subsubsection*{Special cases}
\noindent (i) The generalized gamma distribution associated with Bessel function (see \cite{Sebastian631}) has density
\begin{eqnarray}\label{m4.4}
f_X(x)&=&\frac{a^{\alpha}}{\Gamma(\alpha)e^{\frac{\lambda}{a}}}x^{\alpha-1}e^{-ax}{}_0F_1\left(-;\alpha;\lambda x\right),\ \ \ \  x>0,\nonumber\\
&=&ae^{-\frac{\lambda}{a}}I^{1,0}_{0,1}\Bigg[ax\left|
\begin{matrix}
    \\ 
    (\alpha-1,1,1)
  \end{matrix}
\right.\Bigg]I^{1,0}_{0,2}\Bigg[-\lambda x\left|
\begin{matrix}
    \\ 
    (0,1,1)&(1-\alpha,1,1)
  \end{matrix}
\right.\Bigg],
\end{eqnarray}
where $\alpha,a>0$ and $\lambda\neq 0$.\\
\noindent (ii) The non-central chi-square distribution $\left(\mathrm{put}\ \alpha=\frac{\nu}{2},a=\frac{1}{2}\ \mathrm{and}\ \lambda=\frac{\beta}{4}\ \mathrm{in}\ (\ref{m4.4})\right)$ is
\begin{eqnarray}\label{rt4.4}
f_X(x)&=&\frac{1}{2^{\frac{\nu}{2}}\Gamma(\frac{\nu}{2})}x^{\frac{\nu}{2}-1}e^{-\frac{(x+\beta)}{2}}{}_0F_1\left(-;\frac{\nu}{2};\frac{\beta x}{4}\right),\ \ \ \  x>0,\nonumber\\
&=&\frac{e^{-\frac{\beta}{2}}}{2}I^{1,0}_{0,1}\Bigg[\frac{x}{2}\left|
\begin{matrix}
    \\ 
    (\frac{\nu}{2}-1,1,1)
  \end{matrix}
\right.\Bigg]I^{1,0}_{0,2}\Bigg[-\frac{\beta x}{4}\left|
\begin{matrix}
    \\ 
    (0,1,1)&(1-\frac{\nu}{2},1,1)
  \end{matrix}
\right.\Bigg],
\end{eqnarray}
where $\nu\in\mathbb{N}$ is the degree of freedom and $\beta>0$.\\
\noindent (iii) The generalized gamma distribution associated with confluent hypergeometric function (see \cite{Ghitany727}) has density
\begin{eqnarray}\label{rts4.4}
f_X(x)&=&\frac{a^{\beta}(a+1)^{\alpha-\beta}}{\Gamma(\alpha)}x^{\alpha-1}e^{-(a+1)x}{}_1F_1\left(\beta;\alpha;x\right),\ \ \ \  x>0,\nonumber\\
&=&\frac{a^{\beta}(a+1)^{1-\beta}}{\Gamma(\beta)}I^{1,0}_{0,1}\Bigg[(a+1)x\left|
\begin{matrix}
    \\ 
    (\alpha-1,1,1)
  \end{matrix}
\right.\Bigg]I^{1,1}_{1,2}\Bigg[-x\left|
\begin{matrix}
    (1-\beta,1,1)\\ 
    (0,1,1)&(1-\alpha,1,1)
  \end{matrix}
\right.\Bigg],\nonumber\\
\end{eqnarray}
where $\alpha,\beta,a>0$.
\begin{remark}
For $k=u=v=1,l=\mu=0,|\eta|<1$ and $(c_1,C_1,\gamma_1)=(1,1,1),(d_1,D_1,\rho_1)=(0,1,1)$ the pdf (\ref{nn4.1}) reduces to (\ref{4.1}).
\end{remark}

\section{The $I$-Function inverse Gaussian Distribution}
\setcounter{equation}{0}
\noindent Consider a rv $X>0$ with pdf
\begin{equation}\label{5.1}
f_X(x) =\left\{
	\begin{array}{ll}
	    k_2x^{\alpha-1}e^{-ax-bx^{-1}}I^{m,n}_{p,q}\Bigg[\delta x^{\sigma}\left|
\begin{matrix}
    (a_i,A_i,\alpha_i)_{1,p}\\ 
    (b_j,B_j,\beta_j)_{1,q}
  \end{matrix}
\right.\Bigg], & \mbox{} x>0,\\\\
		0  & \mbox{} \mathrm{otherwise},
	\end{array}
\right.
\end{equation}
with the following restriction on parameters:
\begin{eqnarray*}
&&a>0,\ b\geq 0\ \mathrm{if}\ \alpha> 0,\\
&&a\geq 0,\ b>0\ \mathrm{if}\ \alpha< 0,\\
&&a>0,\ b>0\ \mathrm{if}\ \alpha=0.
\end{eqnarray*}
Further, $\sigma\geq0,\delta\neq 0$ and $(a_i,A_i,\alpha_i)_{1,p},(b_j,B_j,\beta_j)_{1,q},m,n$ must conform to those restrictions in the definition of the $I$-function (\ref{1.2}), such that $f_X(x)\geq 0$. In order to obtain the normalizing constant $k_2$ of the $I$-function inverse Gaussian (I-FIG) distribution, we first determine its Mellin transform.
\begin{theorem}
The Mellin transform of the I-FIG distribution is
\begin{eqnarray*}
&&\mathcal{M}_{f_X}(s)\\
&&=k_2\sum_{l=0}^{\infty}\left\{ a^{-\alpha-s+1}I^{m+1,n+1}_{p+2,q+1}\left.\Bigg[\frac{\delta}{a^{\sigma}}\right|
\begin{matrix}
    (2-\alpha-s,\sigma,1)&(a_i,A_i,\alpha_i)_{1,p}&(2+l-\alpha-s,\sigma,1)\\ 
    (2-\alpha-s,\sigma,1)&(b_j,B_j,\beta_j)_{1,q}
  \end{matrix}
\Bigg]\right.\\
&&\ \ \left.-\ b^{\alpha+s-1}I^{m+1,n+1}_{p+1,q+2}\left.\Bigg[b^\sigma\delta\right|
\begin{matrix}
    (2-\alpha-s,\sigma,1)&(a_i,A_i,\alpha_i)_{1,p}\\ 
    (2-\alpha-s,\sigma,1)&(b_j,B_j,\beta_j)_{1,q}&(1-l-\alpha-s,\sigma,1)
  \end{matrix}
\Bigg]\right\}\frac{(ab)^l}{l!}.
\end{eqnarray*}
\end{theorem}
\begin{proof} Using the definition of Mellin transform and assuming the conditions for absolute convergence of integral involved, we have
\begin{eqnarray}\label{5.2}
\mathcal{M}_{f_X}(s)&=&k_2\int_0^\infty x^{\alpha+s-2}e^{-ax-bx^{-1}}I^{m,n}_{p,q}\left.\Bigg[\delta x^{\sigma}\right|
\begin{matrix}
    (a_i,A_i,\alpha_i)_{1,p}\\ 
    (b_j,B_j,\beta_j)_{1,q}
  \end{matrix}
\Bigg]\,dx\nonumber\\
&=&k_2\int_0^\infty x^{\alpha+s-2}e^{-ax-bx^{-1}}\left\{\frac{1}{2\pi i}\int_{C}\chi(w)(\delta x^{\sigma})^{-w}\,dw\right\}\,dx\nonumber\\
&=&\frac{k_2}{2\pi i}\int_{C}\left\{\int_0^\infty x^{\alpha+s-\sigma w-2}e^{-ax-bx^{-1}}\,dx\right\}\chi(w)\delta^{-w}\,dw.
\end{eqnarray}
Now
\begin{equation}\label{5.3}
\int_0^\infty x^{\alpha+s-\sigma w-2}e^{-ax-bx^{-1}}\,dx=2\left(\sqrt{\frac{b}{a}}\right)^{\alpha+s-\sigma w-1}K_{\alpha+s-\sigma w-1}(2\sqrt{ab}),
\end{equation}
where $K_{\nu}(x)$ is the modified Bessel function of the second kind, also called Macdonald function, defined by
\begin{equation}\label{5.4}
K_{\nu}(x)=\frac{\pi\left(I_{-\nu}(x)-I_{\nu}(x)\right)}{2\sin\nu\pi},
\end{equation}
where $I_{\nu}(x)$ is the modified Bessel function of the first kind with the following series representation:
\begin{equation}\label{5.5}
I_{\nu}(x)=\sum_{l=0}^{\infty}\frac{1}{\Gamma(l+\nu+1)l!}\left(\frac{x}{2}\right)^{2l+\nu}.
\end{equation}
Using (\ref{5.3})-(\ref{5.5}) and gamma reflection principle in (\ref{5.2}), we have 
\begin{eqnarray*}
\mathcal{M}_{f_X}(s)&=&\frac{k_2}{2\pi i}\int_{C}\sum_{l=0}^{\infty}\left\{a^{-\alpha-s+1}\frac{\Gamma(\alpha+s-\sigma w-1)\Gamma(2-\alpha-s+\sigma w) \chi(w)}{\Gamma(l-\alpha-s+\sigma w+2)}\left(\frac{\delta}{a^{\sigma}}\right)^{- w}\right.\\
&&\ \ \ \ \left.-b^{\alpha+s-1}\frac{\Gamma(\alpha+s-\sigma w-1)\Gamma(2-\alpha-s+\sigma w)\chi(w)}{\Gamma(l+\alpha+s-\sigma w)}\left(b^{\sigma}\delta\right)^{-w}\right\}\frac{(ab)^l}{l!}\,dw.
\end{eqnarray*}
Finally, by using the definition of the $I$-function (\ref{1.2}), the proof is complete.
\end{proof}
\begin{corollary}
The $r$-th moment for the I-FIG distribution is
\begin{eqnarray*}
\mu_r'&=&k_2\sum_{l=0}^{\infty}\left\{ a^{-\alpha-r}I^{m+1,n+1}_{p+2,q+1}\left.\Bigg[\frac{\delta}{a^{\sigma}}\right|
\begin{matrix}
    (1-\alpha-r,\sigma,1)&(a_i,A_i,\alpha_i)_{1,p}&(1+l-\alpha-r,\sigma,1)\\ 
    (1-\alpha-r,\sigma,1)&(b_j,B_j,\beta_j)_{1,q}
  \end{matrix}
\Bigg]\right.\\
&&\ \ \left.-\ b^{\alpha+r}I^{m+1,n+1}_{p+1,q+2}\left.\Bigg[b^\sigma\delta\right|
\begin{matrix}
    (1-\alpha-r,\sigma,1)&(a_i,A_i,\alpha_i)_{1,p}\\ 
    (1-\alpha-r,\sigma,1)&(b_j,B_j,\beta_j)_{1,q}&(-l-\alpha-r,\sigma,1)
  \end{matrix}
\Bigg]\right\}\frac{(ab)^l}{l!}.
\end{eqnarray*}
Hence, the normalizing constant of the I-FIG distribution is given by
\begin{eqnarray*}
k_2^{-1}&=&\sum_{l=0}^{\infty}\left\{ a^{-\alpha}I^{m+1,n+1}_{p+2,q+1}\left.\Bigg[\frac{\delta}{a^{\sigma}}\right|
\begin{matrix}
    (1-\alpha,\sigma,1)&(a_i,A_i,\alpha_i)_{1,p}&(1+l-\alpha,\sigma,1)\\ 
    (1-\alpha,\sigma,1)&(b_j,B_j,\beta_j)_{1,q}
  \end{matrix}
\Bigg]\right.\\
&&\ \ \ \ \ \ \ \ \ \left.-\ b^{\alpha}I^{m+1,n+1}_{p+1,q+2}\left.\Bigg[b^{\sigma}\delta\right|
\begin{matrix}
    (1-\alpha,\sigma,1)&(a_i,A_i,\alpha_i)_{1,p}\\ 
    (1-\alpha,\sigma,1)&(b_j,B_j,\beta_j)_{1,q}&(-l-\alpha,\sigma,1)
  \end{matrix}
\Bigg]\right\}\frac{(ab)^l}{l!}.
\end{eqnarray*}
\end{corollary}
\begin{corollary}
The Laplace transform of the I-FIG distribution is
\begin{equation*}
\mathcal{L}_{f_X}(r)=k_2\sum_{l=0}^{\infty}\left\{\frac{1}{(a+r)^{\alpha}}I^{m+1,n+1}_{p+2,q+1}\left.\Bigg[\frac{\delta}{(a+r)^\sigma}\right|
\begin{matrix}
    (1-\alpha,\sigma,1)&(a_i,A_i,\alpha_i)_{1,p}&(1+l-\alpha,\sigma,1)\\ 
    (1-\alpha,\sigma,1)&(b_j,B_j,\beta_j)_{1,q}
  \end{matrix}
\Bigg]\right.\\
\end{equation*}
\begin{equation*}
\ \ \ \ \ \ \ \ \ \ \left.-\ b^{\alpha}I^{m+1,n+1}_{p+1,q+2}\left.\Bigg[b^\sigma\delta\right|
\begin{matrix}
    (1-\alpha,\sigma,1)&(a_i,A_i,\alpha_i)_{1,p}\\ 
    (1-\alpha,\sigma,1)&(b_j,B_j,\beta_j)_{1,q}&(-l-\alpha,\sigma,1)
  \end{matrix}
\Bigg]\right\}\frac{\left((a+r)b\right)^l}{l!}.
\end{equation*}
\end{corollary}
\subsubsection*{Special cases}
The generalized inverse Gaussian distribution studied by Barndorff-Nielsen \cite{Nielsen1}.
\begin{eqnarray}\label{n5.6}
f_X(x)&=&\frac{\beta^{\frac{\alpha}{2}}}{2\gamma^{\frac{\alpha}{2}}K_{\alpha}(\sqrt{\beta\gamma})}x^{\alpha-1}e^{-\frac{1}{2}(\beta x+\gamma x^{-1})}\nonumber\\
&=&\frac{\beta^{\frac{\alpha}{2}}}{2\gamma^{\frac{\alpha}{2}}K_{\alpha}(\sqrt{\beta\gamma})}x^{\alpha-1}e^{-\frac{\beta }{2}x-\frac{\gamma}{2}x^{-1}}I^{1,0}_{1,1}\Bigg[\delta\left|
\begin{matrix}
    (1,1,1)\\ 
    (0,1,1)
  \end{matrix}\right.\Bigg],\ \ \ \  x>0,
\end{eqnarray}
where $|\delta|<1,\beta,\gamma>0,\alpha\in\mathbb{R}$. Therefore the inverse Gaussian, inverse gamma and hyperbolic distributions follow as particular cases.\\
(i) The inverse Gaussian distribution $\left(\mathrm{put}\ \alpha=-\frac{1}{2},\beta=\frac{\lambda}{\mu^2}\ \mathrm{and}\ \gamma=\lambda\ \mathrm{in}\ (\ref{n5.6})\right)$.
\begin{eqnarray*}
f_X(x)&=&\sqrt{\frac{\lambda}{2\pi x^3}}e^{-\frac{\lambda(x-\mu)^2}{2\mu^2x}},\\
&=&\sqrt{\frac{\lambda}{2\pi}}e^{\frac{\lambda}{\mu}}x^{-\frac{1}{2}-1}e^{-\frac{\lambda}{2\mu^2}x-\frac{\lambda}{2} x^{-1}}I^{1,0}_{1,1}\Bigg[\delta\left|
\begin{matrix}
    (1,1,1)\\ 
    (0,1,1)
  \end{matrix}\right.\Bigg],\ \ \ \  x>0,
\end{eqnarray*}
where $|\delta|<1$ and $\lambda,\mu>0$. Note that $K_{-\frac{1}{2}}(x)=\sqrt{\frac{\pi}{2}}e^{-x}x^{-\frac{1}{2}}$.\\
(ii) The inverse Gamma distribution.
\begin{eqnarray*}
f_X(x)&=&\frac{\lambda^\theta}{\Gamma(\theta)}x^{-\theta-1}e^{-\lambda x^{-1}},\\
&=&\frac{\lambda^\theta}{\Gamma(\theta)}x^{-\theta-1}e^{-\lambda x^{-1}}I^{1,0}_{1,1}\Bigg[\delta\left|
\begin{matrix}
    (1,1,1)\\ 
    (0,1,1)
  \end{matrix}\right.\Bigg],\ \ \ \  x>0,
\end{eqnarray*}
where $|\delta|<1$ and $\theta,\lambda>0$.
\begin{remark}
When $b=0$, the I-FIG distribution in (\ref{5.1}) belongs to the class of distributions defined by the product of two $I$-functions (\ref{nn4.1}) as follows.
\begin{equation*}
f_X(x) =k_*I^{1,0}_{0,1}\Bigg[ax\left|
\begin{matrix}
    \\ 
    (\alpha-1,1,1)
  \end{matrix}
\right.\Bigg]
I^{m,n}_{p,q}\Bigg[\delta x^{\sigma}\left|
\begin{matrix}
    (a_i,A_i,\alpha_i)_{1,p}\\ 
    (b_j,B_j,\beta_j)_{1,q}
  \end{matrix}
\right.\Bigg],\ \ x>0,
\end{equation*}
where
\begin{equation*}
k_*^{-1}=a^{-\alpha}
I^{m,n+1}_{p+1,q}\Bigg[\frac{\delta}{a^{\sigma}}\left|
\begin{matrix}
    (1-\alpha,\sigma,1)&(a_i,A_i,\alpha_i)_{1,p}\\ 
    (b_j,B_j,\beta_j)_{1,q}
  \end{matrix}
\right.\Bigg].
\end{equation*}
Therefore the distributions with pdf (\ref{m4.4})-(\ref{rts4.4}) are also the particular cases of the I-FIG distribution.
\end{remark}
\printbibliography
\end{document}